\documentclass[12pt,reqno, oneside]{amsart}
\usepackage{amsaddr}
\usepackage{hyperref}
\usepackage{pdfsync} 

\newtheorem{lemma}{Lemma}

\def\Atop#1#2{\genfrac{}{}{0pt}{}{#1}{#2}}
\def\F#1#2#3#4{{}_{#1}F_{#2}\left(#3\biggm|#4\right)}
\def\r#1#2{(#1)_{#2}}
\def\half{\tfrac12}

\begin{document}
\title[A simple proof of Andrews's ${}_5F_4$ evaluation]{A simple proof of Andrews's ${}_5F_4$ evaluation}
\author{Ira M. Gessel}

\address{Department of Mathematics\\
   Brandeis University\\
   Waltham, MA 02453\\
   {\tt gessel@brandeis.edu}}
\thanks{This work was partially supported by a grant from the Simons Foundation (\#229238 to Ira Gessel).}

\begin{abstract}
We give a simple proof of George Andrews's balanced ${}_5F_4$ evaluation using two fundamental principles: the $n$th difference of a polynomial of degree less than $n$ is zero, and a polynomial of degree $n$ that vanishes at $n+1$ points is identically zero.
\end{abstract}

\maketitle

\section{Introduction}
George Andrews \cite{pfaff1}, in his evaluation of the Mills-Robbins-Rumsey determinant, needed the balanced $_5F_4$ evaluation
\begin{equation}
\label{e-a1}
\F54{\Atop{-2m-1, x+2m+2, x-z+\half,x+m+1, z+m+1}
{\half x+\half, \half x+1, 2z+2m+2, 2x-2z+1}
}1=0,
\end{equation}
where $m$ is a nonnegative integer. Here the hypergeometric series is defined by
\begin{equation*}
\F pq{\Atop{a_1,\dots, a_p}{b_1,\dots, b_q}}t 
= \sum_{k=0}^\infty \frac{(a_1)_k\cdots(a_p)_k}{k!\,(b_1)_k\cdots(b_q)_k} t^k
\end{equation*}
and $(a)_k$ is the rising factorial $a(a+1)\cdots (a+k-1)$.
Andrews's proof of \eqref{e-a1} used Pfaff's method, and required a complicated induction that proved 20 related identities. Andrews  later discussed these identities and Pfaff's method in comparison with the WZ method \cite{pfaff3}, and a proof of \eqref{e-a1} using the Gosper-Zeilberger algorithm was given by Ekhad and Zeilberger \cite{syndrome}.
A completely different proof of \eqref{e-a1} was given by Andrews and Stanton \cite{a-s}. Generalizations of \eqref{e-a1}, proved using known transformations for hypergeometric series, have been given by Stanton \cite{stanton}, Chu \cite{chu}, and Verma, Jain, and Jain \cite{verma}.

We give here a simple self-contained proof of Andrews's identity, by using two fundamental principles: first, the $n$th difference of a polynomial of degree less than $n$ is~0, and second, a polynomial of degree $n$ that vanishes at $n+1$ points is identically 0.

To illustrate the method, we first use it to prove the Pfaff-Saalsch\"utz identity. We then prove Andrews's identity.

\section{Lemmas}

We first give two lemmas that we will need later on. Although they are well known, for completeness we include the short proofs.

\begin{lemma}
\label{diff}
If $p(k)$ is a polynomial of degree less than $n$ then
\[\sum_{k=0}^n (-1)^k \binom nk p(k) = 0.\]
\end{lemma}
\begin{proof}
Since the polynomials $\binom ki$ form a basis for the vector space of all polynomials in $k$,  it suffices by linearity to show that if $i<n$ then $\sum_{k=0}^n (-1)^k \binom nk \binom ki = 0$.
But
\begin{align*}
\sum_{k=0}^n (-1)^k \binom nk \binom ki &=
(-1)^i\binom ni \sum_{k=i}^n (-1)^{k-i} \binom{n-i}{k-i}=(-1)^i\binom ni (1-1)^{n-i}=0,
\end{align*} 
by the binomial theorem.
\end{proof}

\begin{lemma}
\label{facs}
If $\alpha-\beta=d$ is a nonnegative integer, then $\r \alpha k/\r \beta k$, as a function of $k$, is a polynomial of degree $d$.
\end{lemma}
\begin{proof}
We first note the formula
\[\r u{i+j} = \r ui \r{u+i}j,\]
which we will also use later. 
Then
\[\r \beta d\frac{\r \alpha k}{\r \beta k} = \frac {\r \beta d \r{\beta +d}k}{\r \beta k} 
=\frac{\r \beta {d+k}}{\r \beta k} = \r{\beta +k}d. \qedhere
\]
\end{proof}

We shall also use the fact that a polynomial of degree at most $d$ is determined by its value at $d+1$ points, or by its leading coefficient and its value at $d$ points.

\section{The Pfaff-Saalsch\"utz identity}

As a warm-up we give a proof of the Pfaff-Saalsch\"utz identity
\begin{equation}
\label{e-pfaff}
\F32{\Atop{-m,\ a,\ b}{c, 1-m+a+b-c}}1 = \frac{\r{c-a}m \r{c-b}m}{\r cm \r{c-a-b}m}.
\end{equation}
We assume that $a-b$ is not an integer; it is easy to see that the identity with this restriction  implies the general case.
First we show that the left side of \eqref{e-pfaff} vanishes if $c-a\in \{0, -1, \dots, -(m-1)\}$. With $c-a=-i$, we may write the left side of \eqref{e-pfaff} as
\begin{equation}
\label{e-psum}
\sum_{k=0}^m (-1)^k \binom mk \frac{\r {c+i}k}{\r ck} \frac{\r bk}{\r {1-m+i+b}k}.
\end{equation}
By Lemma \ref{facs}, 
\[\frac{\r {c+i}k}{\r ck} \frac{\r bk}{\r {1-m+i+b}k}\]
is a polynomial in $k$ of degree $i + (m-i-1) = m-1$, so by Lemma \ref{diff}, the sum \eqref{e-psum} vanishes.
By symmetry, \eqref{e-psum} also vanishes if $c-b\in \{0, -1, \dots, -(m-1)\}$.

Multiplying the left side of \eqref{e-pfaff} by $\r cm \r {c-a-b}m$ and simplifying gives
\begin{multline}
\label{e-saalpoly}
\qquad \r cm \r {c-a-b}m\,
\F32{\Atop{-m,\ a,\ b}{c, 1-m+a+b-c}}1 \\
=
\sum_{k=0}^m  \binom mk \r ak \r bk \r{c+k}{m-k} \r{c-a-b}{m-k}.\qquad
\end{multline}
Then \eqref{e-saalpoly} is a monic polynomial in $c$ of degree $2m$ that  vanishes for the $2m$ distinct (since $a-b$ is not an integer) values $c=a-i$ and $c=b-i$, for $i\in \{0, 1, \dots m\}$. Thus \eqref{e-saalpoly} is equal to $(c-a)_m(c-b)_m$.
\smallskip

We note that the sum in the Pfaff-Saalsch\"utz theorem is \emph{balanced}; that is, the sum of the denominator parameters is one more than the sum of the numerator parameters. It is not hard to show that if a balanced hypergeometric series can be expressed in the form 
\[\sum_{k=0}^m (-1)^k\binom mk p(k),\]
where $p(k)$ is a polynomial in $k$, then $p(k)$ must have degree $m-1$, and thus the sum vanishes by Lemma \ref{diff}. For this reason, our method is especially applicable to balanced summation formulas. 

\section{Andrews's Identity}
To prove \eqref{e-a1}, we start by making the substitution $x=y+2z$, obtaining the equivalent identity 
\begin{equation}
\label{e-a2}
\F54{\Atop{-2m-1, y+2z+2m+2, y+z+\half,y+2z+m+1, z+m+1}
{\half y+z+\half, \half y+z+1, 2z+2m+2, 2y+2z+1}
}1=0.
\end{equation}
We shall first show that \eqref{e-a2} holds when $y\in \{0,1,\dots, 2m+1\}$ by applying Lemma~\ref{diff}. We will then derive the general result by expressing the sum as a polynomial in $y$ of degree $2m$.

\begin{lemma}
\label{adiff}
Formula \eqref{e-a2} holds for $y\in \{0,1,\dots, 2m+1\}$.
\end{lemma}

\begin{proof}
We write the sum in \eqref{e-a2} as
\begin{equation*}
\sum_{k=0}^{2m+1}(-1)^k\binom {2m+1}k  P_1(k) P_2(k),
\end{equation*}
where 
\[P_1(k) = \frac{\r {y+2z+2m+2}k \r{y+2z+m+1}k} {\r{2z+2m+2}k \r{2y+2z+1}k}\]
and
\[P_2(k) = \frac{\r{y+z+\half}k \r{z+m+1}k} {\r {\half y+z+\half}k \r{\half y+z+1}k}.\]
\end{proof}
It will suffice to show that for each $y\in \{0,1,\dots, 2m+1\}$, $P_1(k)$ and $P_2(k)$ are polynomials in $k$. We  do this by pairing up the numerator and denominator factors in $P_1(k)$ and $P_2(k)$ so that Lemma \ref{facs} applies.

For $0\le y\le m$ we use
\[P_1(k) = \frac{\r {y+2z+2m+2}k}{\r{2z+2m+2}k}\cdot \frac{\r{y+2z+m+1}k}{\r{2y+2z+1}k},\]
and for $m+1\le y\le 2m+1$, we use
\[P_1(k) =  \frac{\r {y+2z+2m+2}k}{\r{2y+2z+1}k}\cdot \frac{\r{y+2z+m+1}k}{\r{2z+2m+2}k}.\]
For $y$ even, we use
\[P_2(k) = \frac{\r{y+z+\half}k}{\r{\half y +z +\half}k}
 \cdot \frac{\r{z+m+1}k}{\r{\half y+z+1}k},\]
 and for $y$ odd we use
 \[P_2(k) = \frac{\r{y+z+\half}k}{\r{\half y+z+1}k}
 \cdot \frac{\r{z+m+1}k}{\r{\half y +z +\half}k}.\]
It is easily checked that Lemma \ref{facs} applies in all cases. So for each $y$, $P_1(k)P_2(k)$ is a polynomial in $k$ of degree $2m$, and the result follows from Lemma \ref{diff}.

\begin{lemma}
\label{poly} The series in \eqref{e-a2}, after multiplication by $\r{y+z+1}m \r{y+2z+1}m$, is a polynomial in $y$ of degree at most $2m$.
\end{lemma}
\begin{proof} 
We show that each term in the sum, when multiplied by $\r{y+z+1}m \r{y+2z+1}m$, is a polynomial in $y$ of degree at most $2m$. Ignoring factors that do not contain $y$, we see that we must show that for $0\le k\le 2m+1$, 
\begin{equation*}
 \r{y+z+1}m \r{y+2z+1}m \frac{\r{y+2z+2m+2}k \r{y+z+\half}k \r{y+2z+m+1}k}
 {\r{\half y+z+\half}k \r{\half y+z+1}k \r{2y+2z+1}k}
\end{equation*}
is a polynomial in $y$ of degree at most $2m$.
To do this we define
\begin{equation}
\label{e-yz}
Q_1(y)= \r{y+z+1}m \frac{\r{y+z+\half}k }
 {\r{2y+2z+1}k}
\end{equation}
and 
\begin{equation}
\label{e-y2z}
Q_2(y)=\r{y+2z+1}m \frac{\r{y+2z+2m+2}k \r{y+2z+m+1}k}
 {\r{\half y+z+\half}k \r{\half y+z+1}k }
 \end{equation}
and we show that $Q_1(y)$ and $Q_2(y)$ are both polynomials in $y$ of degree $m$.
We will use the formulas
\begin{align*}
\r a{2n} &= 2^{2n}\r{\half a}n \r{\half a+\half}n, \\
\r a{2n+1}&=2^{2n+1}\r{\half a}{n+1}\r{\half a + \half}n.
\end{align*}

For $k\le m$, we have 
\begin{align*}
Q_1(y) &= \r{y+z+1}k \r{y+z+1+k}{m-k} 
\frac{\r{y+z+\half}k} {\r{2y+2z+1}k}\\
  &=2^{-2k}\r{y+z+1+k}{m-k}
     \frac{\r{2y+2z+1}{2k}}{\r{2y+2z+1}k} \\
  &=2^{-2k}\r{y+z+1+k}{m-k}
    \r{2y+2z+1+k}{k},
\end{align*}
and for $m+1\le k\le 2m+1$ we have
\begin{align*}
Q_1(y)&= \r{y+z+1}m\frac{\r{y+z+\half}{m+1}\r{y+z+m+\tfrac32}{k-m-1}}
{\r{2y+2z+1}k}\\
   &= 2^{-2m-1}\frac{\r{2y+2z+1}{2m+1}}{\r{2y+2z+1}k}\r{y+z+m+\tfrac32}{k-m-1}\\
   &=2^{-2m-1}\r{2y+2z+1+k}{2m+1-k}\r{y+z+m+\tfrac32}{k-m-1},
\end{align*}
so in both cases, $Q_1(y)$ is a polynomial in $y$ of degree $m$.

We have
\begin{equation*}
Q_2(y) = 2^{2k} \frac{\r{y+2z+1}{m+k}\r{y+2z+2m+2}k}{\r{y+2z+1}{2k}}.
\end{equation*}
For $k\le m$ we have
\[Q_2(y)=2^{2k}\r{y+2z+1+2k}{m-k}\r{y+2z+2m+2}k.\]
For $m+1\le k\le 2m+1$, we have
\begin{align*}
Q_2(y) &= 2^{2k} \frac{ \r {y+2z+1}{m+k} \r {y+2z+2m+2}k}{\r{y+2z+1}{2k}}\\
  &=2^{2k} \frac{ \r {y+2z+1}{m+k}}{\r{y+2z+1}{2m+1}} \cdot \frac{\r{y+2z+1}{2m+1}   \r {y+2z+2m+2}k}{\r{y+2z+1}{2k}}\\
  &=2^{2k} \r {y+z+2m+2}{k-m-1} \frac{\r{y+2z+1}{2m+1+k}}{\r{y+2z+1}{2k}}\\
  &= 2^{2k} \r {y+z+2m+2}{k-m-1}\r{y+2z+1+2k}{2m+1-k}.
\end{align*}
Thus in both cases, $Q_2(y)$ is also a polynomial in $y$ of degree $m$.

As an alternative, we could have expressed $Q_1(y)$ and $Q_2(y)$ as rising factorials with respect to $y$, 
\begin{align*}
Q_1(y) &=  C_1\frac{\r{z+m+1}y \r{z+k+\half}y}{\r{z+\half k +\half}y \r{z+\half k +1}y}\\
Q_2(y) &= C_2 \frac{\r{2z+m+k+1}y \r{2z+2m+k+2}y}{\r {2z+2k+1}y \r{2z+2m+2}y},
\end{align*}
where $C_1$ and $C_2$ do not contain $y$,
and applied Lemma \ref{facs}.
\end{proof}

We can now finish the proof of \eqref{e-a2}. By Lemmas \ref{adiff} and \ref{poly}, $\r{y+z+1}m \r{y+2z+1}m$ times the sum in \eqref{e-a2} is a polynomial in $y$ of degree at most $2m$ that vanishes for $y=0, 1,\dots, 2m+1$. Therefore this polynomial is identically zero.

\end{document}